\documentclass{article}
\usepackage{tikz}
\usepackage{amsmath,amsfonts,amsthm,amssymb,latexsym,amscd,mathtools,enumerate}
\usepackage{xspace}
\input amssym.def
\input amssym.tex
\theoremstyle{definition}
\newtheorem{definition}{Definition}
\def\newline{\hfill\break}
\def\scong{{\scriptstyle\|}\lower.2ex\hbox{$\wr$}}
\def\Z{{\Bbb Z}}

\def\Q{{\Bbb Q}}

\def\End{\mathop{\rm End}\nolimits}

\def\Hom{\mathop{\rm Hom}\nolimits}

\def\Spec{\mathop{\rm Spec}\nolimits}

\def\rtimes{\mathop{\times\!\!{\raise.2ex\hbox{$\scriptscriptstyle|$}}}
	\nolimits} 

\outer\def\Demo #1. #2\par{\medbreak\noindent {\it#1.\enspace}
	{\rm#2}\par\ifdim\lastskip<\medskipamount\removelastskip
	\penalty55\medskip\fi}

\def\P{{\Bbb P}}

\def\hangbox to #1 #2{\vskip1pt\hangindent #1\noindent \hbox to #1{#2}$\!\!$}
 
\title{Stable Rationality and Cyclicity}
\author{David J. Saltman\\  
Center for Communications Research\\
805 Bunn Drive\\
Princeton, NJ 08540}

\newtheorem{theorem}{Theorem}[section]
\newtheorem{corollary}[theorem]{Corollary}
\newtheorem{lemma}[theorem]{Lemma}
\newtheorem{proposition}[theorem]{Proposition}

\begin{document}
\maketitle

\section{Introduction} 
The subject of division algebras has a number of issues 
where we understand small degrees but cannot settle 
the general case. The most prominent such issue 
concerns division algebras of prime degree and whether 
they were cyclic algebras or equivalently crossed products. 
Amitsur (\cite{A}) showed there were non-crossed product division algebras, 
but his method very strongly relied on the degree 
being composite. Division algebras of degree 2 and 3 are 
known to be cyclic (that is a crossed product with cyclic 
group)(\cite{Al} p. 177). For $p > 3$ we just do not know. 

In his non-crossed product construction 
Amitsur made use of the so called generic division algebra 
written $UD(F,n)$. The center of this division algebra, 
written $Z(F,n)$, is important both for its control 
of the behavior of $UD(F,n)$ and because 
$Z(F,n)$ arises as the invariant field of the 
linear group $PGL_n(F)$ acting on two copies of 
matrices $M_n(F) \oplus M_n(F)$. Of course $Z(F,n)$ 
is interesting for all $n$ but we focus on $Z(F,p)$ 
for a prime $p$. Once again we have results 
for small $p$. $Z(F,2)/F$ and $Z(F,3)/F$ are known 
to be rational (i.e. purely transcendental) field 
extensions (\cite{P} and \cite{F}). 
$Z(F,5)/F$ and $Z(F,7)/F$ are stably rational
(\cite{BL}). 
Here, however, we do have a general result. 
For any prime $p$ we know that $Z(F,p)/F$ is retract 
rational (\cite{S} p. 116). Thus the issue here is that it is  difficult to 
detect the difference between stably rational field extensions
and the weaker retract rational field extensions. 

The purpose of this paper is to describe a relationship 
between these two issues. Specifically we will show 
(\ref{result}) that if $UD(F,p)$ is cyclic then $Z(F,p)/F$ 
is stably rational. To be more provocative we 
showcase the contrapositive, namely, that 
if $Z(F,p)$ is not stably rational then 
$UD(F,p)$ is not cyclic. Note that if $UD(F,p)$ is cyclic 
then for all $D/K$ degree $p$ with $K \supset F$ 
$D/K$ is cyclic. 

In this paper $F$ will always have characteristic 0 and will always contain 
$\rho$, a primitive $p$ root of one for an odd prime $p$ 
fixed throughout the paper. (There is nothing to prove in the $p = 2$ case). The purpose of the characteristic 0  assumption is to make geometry easier but is likely not necessary. In particular, let $X$ and $Y$ be irreducible varieties over a field $F$ and 
$\phi: Y \to X$ a dominant morphism. 
If the generic fiber of $\phi$ is finite, 
$\phi$ is generically finite and generically etale. If $F$ is assumed to be algebraically closed, 
then the degree of $\phi$ is the cardinality of this generic fiber. 
Moreover,  
$\phi(Y)$ contains a dense open subset of $X$. 

A division algebra $D/F$ is an 
algebra where all nonzero elements  
have a multiplicative inverse 
and which is finite dimensional over its center $F$. 
This dimension is always of the form $n^2$ and $n$ is the 
{\bf degree} of the division algebra. If $\tilde F$ is the 
algebraic closure of $F$ then $D_{\tilde F} = 
D \otimes_F \tilde F \cong M_n(\tilde F)$, where $M_n(\tilde F)$ is 
the matrix algebra over $\tilde F$. The determinant 
map on $M_n(\tilde F)$ descends to the {\bf reduced norm} 
$n: D \to F$. 

When $F$ is as assumed here, cyclic division 
algebras have a special form, often called "symbol algebras". 
If $x,y \in F^*$ we let $(x,y)_{p,F}$ be the algebra 
generated by $\gamma$, $\delta$ subject to the 
relations $\gamma^p = x$, $\delta^p = y$, and 
$\gamma\delta = \rho\delta\gamma$. We say $\gamma$ and 
$\delta$ {\bf skew commute}. This symbol algebra 
is always central simple and has degree $p$. 
If $D/F$ is cyclic of degree $p$, then $D \cong (x,y)_{p,F}$ 
for some $x,y \in F^*$. 

The important thing about symbol algebras is that we have 
two ways to create new skew commuting pairs. If 
$f(x) \in F[x]$ is a polynomial of degree less than that 
$p$ (we write $f(x) \in F[x]_p$) then $\gamma,f(\gamma)\delta$ 
and $f(\delta)\gamma,\delta$ both skew commute. 
The $p$ powers change. $(f(\gamma)\delta)^p = 
n_{F(\gamma)/F}(f(\gamma))\delta^p$ and $(f(\delta)\gamma)^p = 
n_{F(\delta)/F}(f(\delta))\gamma^p$ where 
$n_{K/F}: K \to F$ is the norm map 
of the field extension. Of course if 
$f(\gamma)$ or $f(\delta)$ happen to have norm one 
the $p$ powers do not change. 

In this paper we will often write a nonsingular 
$n \times n$ 
matrix $A$ as $(\vec v_0,\ldots,\vec v_{n-1})$ 
where the $\vec v_i$ are the columns of $A$ and we number the 
columns $0,\ldots,n-1$. When we do present $A$ 
as a matrix $(v_{ij})$ we also number the rows of $A$ 
starting at $0$. 

For $X$ a variety let $k(X)$ be the function field. 
If $X$ is defined over $F$ and $K/F$ is an extension field 
set $X_K = X \times_{\Spec F} \Spec K$ 
which we will usually write as $X \times_F K$. 

A key tool in this paper is the Severi-Brauer variety 
of a central simple algebra. This construction 
is described in many places and a convenient one for this 
paper is Chapter 13 of \cite{S}. One result from 
\cite {S} p. 96 will be of particular importance for us. 

\begin{theorem}\label{isomo}
Let $E/F$ and $E'/F$ be central simple algebras over 
$F$ of equal degrees. Let $X = SB(E' \otimes_F E^{\circ})$ 
be the Severi Brauer variety of $E' \otimes_F E^{\circ}$. 
Then $X$ is birationally isomorphic to to the variety 
of algebra isomorphisms $E \cong E'$. 
\end{theorem} 

Using the above we showed in \cite{S} p. 113

\begin{theorem}\label{rational}
Let $E/F$ be a central simple algebra of degree 
$n$. Form $E' = E \otimes_F Z(F,n)$. 
Let $X = SB(UD(F,n) \otimes_{Z(F,n)} (E')^{\circ})$. 
Then $X$ is rational over $F$. 
\end{theorem}

\section{Some varieties}

Let $D$ be a division algebra of odd prime degree $p$ 
over a field $F$ of characteristic 0 containing $\rho$, 
a primitive $p$ root of unity. Often we will assume 
$D/F$ is cyclic, so $D \cong (x,y)_{p,F}$ 
where $(x,y)_{p,F}$ is generated by $\gamma$, $\delta$ 
subject to $\gamma^p = x$, $\delta^p = y$, and 
$\gamma\delta = \rho\delta\gamma$. 

We define a series of varieties, starting with 
$P(D) = P \subset D^* \times D^*$ defined by the 
equations underlying $\alpha\beta - \rho\beta\alpha = 0$. 
That is, $P$ is the variety of skew commuting pairs.  
More precisely, let $u_1,\ldots,u_{p^2}$ be a basis 
for $D$ over $F$ where it is convenient to assume 
$u_1 = 1$. Write $\alpha = \sum_i x_iu_i$ and 
$\beta = \sum_i y_iu_i$ and $\alpha\beta - 
\rho\beta\alpha = \sum_k f_k(\vec x,\vec y)u_k$ 
where the $f_k(\vec x,\vec y)$ are homogeneous 
polynomials in the $x_i$ and $y_j$. 
We let $I \subset F[\vec x,\vec y]$ be the radical 
of the ideal generated by all the $f_k$. Then 
$P$ is $\Spec(F[\vec x,\vec y]/I)$. 

Note that the relation $\alpha\beta = \rho\beta\alpha$ 
is separately homogenous in $\alpha$ and $\beta$. 
Thus the $f_k(\vec  x,\vec y)$ are separately homogeneous 
in the $x_i$ and $y_i$. Hence it makes 
sense to define $\bar P = \bar P(D) \subset \P(D^*) \times \P(D^*)$ defined 
as the zeroes of $I$. There is a natural 
morphism $P \to \bar P$. One can think of 
$P$ as a double cone over $\bar P$. 

\begin{lemma}\label{bihomogeneous}
If $C \subset P$ is Zariski closed and closed 
under the action of $F^* \times F^*$ then 
the image of $C$ in $\bar P$ is Zariski closed. 
\end{lemma}

\begin{proof} Let $R = F[\vec x,\vec y]/I$ be the affine ring of $P$ and 
$J \subset F[\vec x,\vec y]$ 
the ideal defining $C$. Then $J$ is closed under the action 
of $F^* \times F^*$ and must be homogeneous separately 
in the $x_i$ and $y_i$. Thus $J$ defines the closed image 
of $C$ in $\bar P$. 
\end{proof} 

We now define a third variety. Let $\hat P \subset D^* \times D^*$ be the affine 
variety defined as the set of $\{(\alpha,\beta) | \alpha\beta = \rho\beta\alpha, 
\alpha^p = 1, \beta^p = 1\}$. That is, just as above, these relations define 
$3p^2$ relations $g_i(\vec x,\vec y) = 0$, $J$ is the radical of the ideal generated by 
the $g_i$ and $\hat P$ is the affine variety in $D^* \times D^*$ defined by $J$. Thus $\hat P \subset P$ is closed. Let $\tilde F$ be the algebraic 
closure of $F$. 

\begin{lemma}
$\hat P$ is absolutely irreducible of dimension $p^2 - 1$. 
\end{lemma}

\begin{proof}
$\hat P(D) \times_F \tilde F = \hat P(D \otimes_F \hat F)$ 
and $D \otimes_F \tilde F \cong M_p(\tilde F)$. 
If $(\alpha,\beta)$ and $(\alpha',\beta')$ are $\tilde F$ points of
$\hat P(M_p(\tilde F))$ then they are conjugate. That is, 
$\hat P(M_p(\tilde F)) \cong 
M_p(\tilde F)^*/\tilde F^*$ 
which shows $\hat P \times_F \tilde F$ is irreducible.
\end{proof}

We are interested in the composition 
$\hat Q: \hat P \to P \to \bar P$. 

\begin{proposition}\label{Galois}
$\hat Q$ is a Galois cover with group $C_p \times C_p$. 
\end{proposition}

Note that $\hat P$ has $F$ no rational 
points unless $D$ is split.

\begin{proof}
$\hat Q$ is surjective on $\tilde F$ points. 
Given $(\alpha,\beta)$ in $\bar P$, let $a = \alpha^p \in \tilde F^*$ 
and $b = \beta^p \in \tilde F^*$. Then $(\alpha/a^{1/p},\beta/b^{1/p})$ 
equals $(\alpha,\beta)$ in $\bar P$ and is obviously in the image of 
$\hat Q$. It follows that $\hat Q$ is dominant. This shows: 

\begin{lemma}
$\bar P$ and hence $P$ are absolutely irreducible.
\end{lemma}

For each $u_i$ and $u_j$ we get an 
affine open in $V_i \times V_j \subset \P(D^*) \times \P(D^*)$ 
defined by $x_i \not= 0$ and $y_j \not= 0$. 
If $(V_i \times V_j) \cap \bar P$ is empty we can ignore 
it. If not $(V_i \times V_j) \cap \bar P = U_{ij}$ is nonempty 
open, affine and irreducible. In fact, $U_{ij}$ 
can be identified with $U_{ij} \subset P$ defined by 
$x_i = 1$ and $y_j = 1$. 

There are generic $\alpha_i,\beta_j$ 
defined on $U_{ij}$ with 
$\alpha_i\beta_j = \rho\beta_j\alpha_i$ and then 
$\alpha_i^p = a_iu_1$ and $\beta_j^p = b_ju_1$ 
for $a_i,b_j$ global units on $U_{ij}$. 
On $U_{ij} \cap U_{i'j'}$ we have $\alpha_i = \alpha_{i'}c_{ii'}$ 
and $\beta_j = \beta_{j'}d_{jj'}$ for $c_{ii'}, d_{jj'} \in F^*$. 
Thus $a_i = a_{i'}c_{ii'}^p$ and $b_j = b_{j'}d_{jj'}^p$. 

We can thus define $X = P(a_i^{1/p},b_j^{1/p})$ which is 
independent of the choice of $i,j$ and is $C_p \oplus C_p$ 
Galois over $P$. On $U_{ij}(a_i^{1/p},b_j^{1/p})$ are defined $\hat \alpha_i = \alpha_i/(a_i^{1/p})$ 
and $\hat \beta_j = \beta_j/b_j^{1/p}$ and this defines 
$U_{ij}(a_i^{1/p},b_j^{1/p}) \to \hat P$. 
Clearly $\hat \alpha_i$ and $\hat \alpha_{i'}$ differ by a root 
of unity (and similarly for the $\beta$'s) which is already 
ambiguous in our definitions. It is easy to see that 
these maps patch together 
to define $X \to \hat P$. Going to $\tilde F$ makes it clear 
this is an isomorphism. This proves \ref{Galois}
\end{proof}

\begin{corollary} 
Despite appearances, $\bar P$ is affine. It has dimension 
$p^2 - 1$. $P$ has dimension $p^2 + 1$. 
\end{corollary}

Thus all the varieties 
we are defining are affine. Closed subvarieties 
of affines are affine, so all the varieties we define in this 
paper are affine. This means that when we have finite group 
actions we can define quotients without worry. 
 
In summary, these three varieties fit into 
a diagram: 
$$\begin{matrix}
P\cr
\downarrow\cr
\bar P&\longleftarrow&\hat P
\end{matrix}$$ 
and there is a corresponding diagram over $\tilde F$ 
$$\begin{matrix}
P_{\tilde F}\cr
\downarrow\cr
\bar P_{\tilde F}&\longleftarrow&\hat P_{\tilde F}
\end{matrix}$$
Here the vertical arrows are the double cones over 
$\bar P$ and $\bar P_{\tilde F}$. The horizontal arrows are $C_p \oplus C_p$ 
Galois extensions. 

Now we can begin to outline the argument of this paper. 
Our goal will be to show that $P$ is a rational variety. 
To do this we will, in the third section, define a series of subvarieties 
$P_2 \subset P_{3} \subset \ldots \subset P_{p+1} = P$ 
and show all the $P_i$ are rational. To accomplish this we will 
define corresponding towers $\bar P_2 \subset \ldots \subset \bar P_{p+1} = \bar P$ and 
$\hat P_2 \subset \ldots \subset \hat P_{p+1} = \hat P$. 
To understand the first tower we will extend scalars to $\tilde F$ 
and in that context understand the third tower. In section three we will also explain why it is natural to begin our 
towers at 2. 

We observe next how two of the varieties defined above have alternate birational descriptions. 
The second description, of $\hat P_{\tilde F} = \hat P \times_F \tilde F$, 
will begin to make clear why the strategy outlined above works. 

But first we begin with 
$P$. 

We can view $D \otimes_F k(P)$ as the 
extension making $D$ cyclic in a generic way. There is another approach to the same end. 
Let $F(x,y) = F'$ be the purely transcendental extension. 
Over $F'$ form the "generic" cyclic algebra $E = (x,y)_{p,F'}$ 
generated by $\gamma$, $\delta$ 
such that $\gamma^p = x$, $\delta^p = y$, and $\gamma\delta = 
\rho\delta\gamma$. Next let 
$D_{F'} = D \otimes_F F'$ which we write as $D'$. Set 
$X = SB(E \otimes_{F'} (D')^{\circ})$ to be the Severi Brauer variety and let 
$K = k(X)$ be the corresponding generic splitting field. Of course, $D \otimes_F K 
\cong E \otimes_{F'} K$ is a cyclic algebra. 

\begin{theorem}\label{iso}
$K$ above is isomorphic to 
$k(P)$. 
\end{theorem}

\begin{proof}
By \ref{isomo}, $X$ is birationally the variety of isomorphisms 
$E \to D_{F'}$. Let 
$\phi: E \otimes_{F'} K 
\to D \otimes_F K$ correspond to the generic point of $X$. 
Write $\phi(\gamma) = \sum_i w_iu_i$ and $\phi(\delta) = \sum_i z_iu_i$ where all the $w_i$, $z_i$ 
are in $K$. Let $I \subset 
F[\vec x,\vec y]$ be the ideal defining $P$. Since $\gamma\delta = \rho\delta\gamma$ all the polynomials in $I$ are 0 on 
$\vec w = (w_i)$ and $\vec z = (z_i)$. 

Since $x = \gamma^p$ it follows 
that $(\sum_i w_iu_i)^p = 
xu_1$ and so $x$ is a polynomial in the $w_i$. Similarly, $y$ 
is a polynomial in the $z_i$. 
Thus $F' = F(x,y) \subset 
F(\vec w,\vec z)$ and 
$F'(\vec w,\vec z) = F(\vec w,\vec z)$. Since $F'(\vec w,\vec z)$ 
is the generic point it has transcendence degree $p^2 - 1$ 
over $F'$ and hence degree $p^2 + 1$ over $F$. Since $P$ has dimension $p^2 + 1$, 
$F[\vec w,\vec z)] \cong F[\vec x,\vec y]/I$ and $K$ is the field 
of fractions of $F[\vec x,\vec y]/I$.
\end{proof}

Next we turn to giving an alternate description of $\hat P_{\tilde F}$. 
We have $D_{\tilde F} \cong \End_{\tilde F}(V)$ for a $\tilde F$ 
vector space $V$. Let ${\cal B}$ 
be the variety over $\tilde F$ of ordered bases 
$(\vec v_0,\ldots,\vec v_{p-1})$ 
of $V$, scaled by $\tilde F^*$. Of course ${\cal B}$ is nothing but $PGL_p(\tilde F)$ but it is useful to think of it as ordered bases. 

\begin{theorem} 
There is an isomorphism $\Phi: {\cal B} \cong \hat P_{\tilde F}$ described in the proof below. 
\end{theorem}

\begin{proof}
Given a basis $v_0,\ldots,v_{p-1}$, we define 
$\hat \alpha \in \End(V)$ by $\hat \alpha(v_i) = 
\rho^i{v_i}$ and $\hat \beta(v_i) = v_{i+1}$ 
where we take the index modulo $p$. Then $\Phi((v_0,\ldots,v_{p-1})) = 
(\hat \alpha,\hat \beta)$. 

Given a pair $\hat \alpha, \hat \beta$ in 
$\hat P$, we note that $\hat \alpha$ is forced to be 
separable and so $\hat \alpha$ has one dimensional 
eigenspaces $L_i$ where $\hat \alpha(v) = \rho^i{v}$ 
for all $v \in L_i$. It follows that $\hat \beta(L_i) = 
L_{i+1}$. We can choose $v_0$ arbitrarily using the 
scaling in ${\cal B}$ and we can set $v_i = \beta^i(v_0)$. 
This defines the inverse of $\Phi$. 
\end{proof} 

$\hat P$ has automorphisms corresponding to the 
generators of the Galois group $C_p \oplus C_p$ which we now  describe. 
On $U_{ij}$ we have $\hat \alpha_i/a_i^{1/p}$ and $\hat \beta_j = \beta_j/b_j^{1/p}$ 
and $C_p \oplus C_p$ acts trivially on 
$\bar P$. This the Galois group acts by 
changing $\hat \alpha_i$ and $\hat \beta_j$ by $p$ roots of unity. These actions patch and so we have:  

\begin{lemma} 
Let $\sigma$ be the 
automorphism of $\hat P$ 
defined by $\sigma(\hat \alpha) = 
\rho\hat \alpha$, $\sigma(\hat \beta) = 
\hat \beta$ while $r(\hat \alpha) = \hat \alpha$ and 
$r(\hat \beta) = \rho\hat\beta$. 
Then $\sigma$ and $r$ generate the Galois 
group $C_p \oplus C_p$ of $\hat P/\bar P$.
\end{lemma}

The reason for the $r$, $\sigma$ notation will be clear shortly. 

The isomorphism $\Phi$ can be used to define corresponding maps $\sigma$ 
and $r$
on ${\cal B}$. We need the detailed description of these actions. 
If $\Phi((\vec v_0,\ldots,\vec v_{p-1})) = (\hat \alpha, \hat \beta)$, then 
$\sigma(\hat \alpha)(\vec v_i) = \rho\hat \alpha(\vec v_i) = 
\rho^{i+1}v_i$ and $\hat \beta(\vec v_i) = v_{i+1}$. 
Thus if we also use $\sigma$ to denote the induced action on ${\cal B}$, we have 
$\sigma((\vec v_0,\ldots,\vec v_{p-1}) = (\vec v_{p-1},\vec v_0,\ldots,\vec v_{p-2})$. We call this map the {\bf shift}.  

Similarly, 
$r(\hat \beta)^i(\vec v_0) = (\rho\hat \beta)^i(\vec v_0) = 
\rho^i\vec v_i$ and so if $r$ 
also denotes the induced action 
on ${\cal B}$, then 
$r(\vec v_0,\ldots,\vec v_{p-1}) = 
(\vec v_0,\rho\vec v_1,\ldots,\rho^{p-1}\vec v_{p-1})$. 
We thus have $\bar P_{\tilde F} 
\cong {\cal B}/<\sigma,r>$. 

In addition, the $\sigma$ and $r$ actions can be 
realized by matrix multiplication on ${\cal B}$. 
Let $A = (\vec v_0,\ldots,\vec v_{p-1})$ 
be the matrix with $i$ column $\vec v_i$. 
Then $r(\vec v_0,\ldots,\vec v_{p-1}) = 
(\vec v_0,\rho\vec v_1,\ldots,\rho^{p-1}\vec v_{p-1})$ 
is the matrix $r(A)$ and we have $r(A) = Ar$ 
where $r$ is the diagonal matrix with $1,\rho,\ldots,\rho^{p-1}$ down the diagonal. 
In a similar way, if 
$\sigma(\vec v_0,\ldots,\vec v_{p-1}) = 
(v_{p-1},v_0,\ldots,v_{p-2})$ is the matrix 
$\sigma(A)$ then $\sigma(A) = A\sigma$ 
where $\sigma$ is the permutation matrix 
with $1$'s in the $i,i+1$ and $(p-1),1$ 
positions and $0$'s everywhere else.

\section{A Tale of Two Tori} 

Ultimately we are going to achieve our theorem by describing an increasing filtration 
$B_2 \subset B_3 \subset \ldots \subset B_{p+1} = {\cal B}$ 
with images 
$(\bar P_2)_{\tilde F} \subset \ldots \subset (\bar P_{p+1})_{\hat F} = 
\bar P_{\hat F}$ which are induced by 
$\bar P_2 \subset \ldots \subset \bar P_{p+1} = \bar P$ 
and $P_2 \subset \ldots \subset P_{p+1} = P$. We will see that each step is realized by operating by an 
alternating sequence of two tori. 

We start with $\hat P$. If $K \supset F$ is an extension field, 
let $K[x]_p$ be the space of 
polynomials over $K$ of degree strictly less than $p$. 
Often we identify $K[x]_p$ with 
$K[x]/(x^p - 1)$ in the obvious way. In particular $K[x]_p$ 
has an automorphism $\tau$ defined by $\tau(x) = \rho^{-1}{x}$. We also
define, for any $f(x) \in   
K[x]_p$, $n(f(x)) = 
\prod_i \tau^i(f(x)) \in K$. 

Suppose $(\hat \alpha, \hat \beta) \in \hat P$ and $g(x) \in \tilde F[x]_p$. Write  $(1/\tau)(g(x)) = 
g(x)/\tau(g(x))$ where, obviously, we are performing our operations in $\tilde F[x]/(x^p - 1)$ and we assume $g(x)$ is invertible making this a rational map. 
Set $\hat \beta' = (1/\tau)(g(x))\hat \beta$. 
Then $(\hat \beta')^p = [(1/\tau)g(\hat \alpha)\hat \beta]^p = 1$ and $\hat \alpha\hat \beta' = \rho\hat \beta'\hat \alpha$. We have defined 
$\hat{\cal T}: \hat P \times \P(\hat F[x]_p) \to \hat P$ via 
$\hat{\cal T}((\hat \alpha,\hat \beta),g(x)) = 
(\hat \alpha,(1/\tau)(g(\alpha))\hat \beta)$. 

Once again, it is useful to detail this 
operation translated to ${\cal B}$. 
Note that $\tau(g(\hat \alpha))\hat \beta = \hat \beta(g(\hat \alpha))$ and so 
$(1/\tau)(g(\hat \alpha))\hat \beta  = 
g(\hat \alpha)\hat \beta{g(\hat \alpha)}^{-1}$. Thus if $\Phi((\vec v_0,\ldots,\vec v_{p-1})) = (\hat \alpha,\hat \beta)$, then 
$\Phi((g(\hat \alpha)\vec v_0,\ldots,g(\hat \alpha)\vec v_{p-1})) = \hat{\cal T}((\hat \alpha, \hat \beta),g(x))$. Since $\hat \alpha(\vec v_i) = 
\rho^i\vec v_i$ we have 
$g(\hat \alpha)\vec v_i) = 
g(\rho^i)\vec v_i$. 

There is an isomorphism $\Theta: \tilde F[x]/(x^p - 1) \cong \oplus_i \tilde F$ given by $\Theta(g(x)) = (g(1),g(\rho),\ldots,g(\rho^{p-1}))$.  
We use $\Phi$ to translate from $\hat P$ 
to ${\cal B}$ and define 
$\hat{\cal T}: {\cal B} \times 
\P(\hat F[x]_p) \to {\cal B}$. 
If we write $(\vec v_0,\ldots,\vec v_{p-1}))$ as the square matrix $A$ then 
$\hat{\cal T}(A,g(x)) = AT_{\vec g}$ 
where $\vec g = (g(1),\ldots,g(\rho^{p-1}))$ 
and $T_{\vec g}$ is the diagonal matrix 
with $\vec g$ down the diagonal. That is, 
the action of 
$\hat{\cal T}$ on ${\cal B}$ is just the action of the diagonal torus $\hat T$. Note that as ${\cal B}$ is scaled by $\tilde F^*$ we can and should view 
$\hat T$ as a torus in $PGL_p(\tilde F)$ of dimension 
$p - 1$. Sometimes we are not precise about the difference between $\hat T$ and the corresponding maximal 
torus in $GL_p(\tilde F)$. 

There is a second toral action we will also need. Again, we start with 
$\hat P$. For convenience, 
we define $\tau': \tilde F[x]/(x^p - 1) \cong \tilde F[x]/(x^p - 1)$ 
by $\tau'(x) = \rho{x}$. Let 
$g(x) \in \hat F[x]/(x^p - 1)$. 
The point is that now 
$\tau'(g(\hat \beta))\hat \alpha = \hat \alpha(g(\hat \beta))$. 
We define $\hat{\cal S}: \hat P \times \P(\hat F[x]_p) \to \hat P$ by 
$\hat{\cal S}((\hat \alpha,\hat \beta),g(x)) = ((1/\tau')(g(\hat \beta))\hat \alpha,\hat \beta)$. This time 
$(1/\tau')(g(\hat \beta))\hat \alpha = 
g(\hat \beta)\hat \alpha{g(\hat \beta)}^{-1}$. Thus if $\Phi(A) = 
\Phi((\vec v_0,\ldots,\vec v_{p-1})) = 
(\hat \alpha, \hat \beta)$ then 
$\Phi(g(\hat \beta)\vec v_0,\ldots,g(\hat \beta)\vec v_{p-1})) = 
\hat{\cal S}((\hat \alpha, \hat \beta),g(x))$. Once again we can view $\hat{\cal S}$ as 
$\hat{\cal S}: {\cal B} \times \P(\hat F[x]_p) \to {\cal B}$. 

Note that if 
$g(x) = \sum_j z_jx^j$ then 
$g(\hat \beta)(\vec v_i) = 
\sum_j z_j\vec v_{i+j}$. 
We see that this is a toral action as follows. Let $R$ be the matrix $(\rho^{-ij})$. That is, $R$ is the matrix with rows and columns indexed $0,\ldots,p-1$ 
and $i,j$ entry $\rho^{-ij}$. 
If $A = (\vec v_0,\ldots,\vec v_{p-1})$ then $AR = (\vec w_0,\ldots,\vec w_{p-1})$ where $\vec w_j = \sum_i \rho^{-ij} \vec v_i$. Let $S_{\vec z}$ be the matrix, with respect to 
the $\vec v_0,\ldots,\vec v_{p-1}$ basis, 
corresponding to the diagonal action of 
$\vec z$ on the $\vec w_0,\ldots,\vec w_{p-1}$ basis.

\begin{proposition}\label{toral}

a) $\hat \beta(\vec w_j) = \rho^j\vec w_j$ and $\hat \alpha(\vec w_j) = \vec w_{j+1}$. 

b) The invertible $g(\hat \beta)$ form a torus, 
$\hat S$,  
diagonal with respect to the $(\vec w_0,\ldots,\vec w_{p-1})$ basis. In detail, if 
$g(x) = \sum_i z_ix^i$ then 
$g(\hat \beta)(\vec w_0,\ldots,\vec w_{p-1})= 
(z_0\vec w_0,\ldots,z_{p-1}\vec w_{p-1})$. 

c) $\hat{\cal T}(\Phi(A),g(x)) = \Phi(AT_{\vec g})$ 
where $g_i = g(\rho^i)$. 

d) $\hat{\cal S}(\Phi(A),g(x)) = \Phi(AS_{\vec z})$ 
where $g(x) = \sum_i z_ix^i$.  

e) If $R' = (\rho^{ij})$ then $RR' = pI$. 

f) $R\hat TR^{-1} = \hat S$. 

\end{proposition} 

\begin{proof}
These are routine computations but for 
illustrative value we prove f). 
Part e) allows us to treat $R'$ as the inverse 
of $R$ because ${\cal B}$ is scaled. 
Let $AS_{\vec x} = (\vec u_0,\ldots,\vec u_{p-1})$.  
Then $AS_{\vec x}R = (\vec z_0,\ldots,\vec z_{p-1})$ 
where $\vec z_k = \sum_j \rho^{-jk} \vec u_j = 
\sum_j \rho^{-jk}(\sum_i x_i\vec v_{i+j})$. 
Writing $\ell = i + j$ this is 
$$\sum_{\ell} v_{\ell}(\sum_i x_i\rho^{-(\ell - i)k}) = 
\sum_{\ell}\rho^{-\ell{k}}\vec v_{\ell}(\sum_i x_i\rho^{ik}) = 
(\sum_i x_i\rho^{ik})\vec w_k.$$ 
That is $AS_{\vec x}R = ART_{\vec y}$ where $y_k = \sum_i x_i\rho^{ik}$. 
\end{proof} 

Again $\hat S$ is a torus in $PGL_p(\hat F)$. 

We observe the easy relationship between the toral actions 
and the $\sigma$, $r$ actions. Recall that for $A \in {\cal B}$, 
$\sigma(A) = A\sigma$ and $r(A) = Ar$. On a vector 
define $\sigma(z_0,\ldots,z_{p-1}) = (z_{p-1},z_0,\ldots,z_{p-2})$. 

\begin{corollary}
If $A \in {\cal B}$, $\sigma(AT_{\vec z}) = \sigma(A)T_{\sigma(\vec z)}$ 
and $r(AS_{\vec z}) = r(A)S_{\sigma(\vec z)}$. 
\end{corollary}

\begin{proof} 
The first statement is nothing more than $A(T_{\vec z}) = 
(A\sigma)(\sigma^{-1}T_{\vec z}\sigma)$. The second statement 
is similar after one notes that $r ^{-1}S_{\vec z}r = 
S_{\sigma(\vec z)}$ by using the $w_j = \sum_i \rho^{-ij}v_i$ 
basis. 
\end{proof} 

It is important to note that the two symmetries we detailed 
above are each members of one of these tori. 
Clearly acting by $r = (1,\rho,\ldots,\rho^{p-1})$ is just acting by 
this element of $\hat T$. Also, acting by $\sigma$ is just acting by 
$\sum_i x_i\hat \beta^i$ where $x_i = 0$ except for 
$x_{p-1} = 1$. We write $r \in \hat T$ and $\sigma \in \hat S$. 
Being a permutation of the basis, 
$\sigma$ clearly normalizes $\hat T$. $ArR = 
(\vec v_0,\rho\vec v_1,\ldots,\rho^{p-1} \vec v_{p-1}) = 
(\vec w_0',\ldots,\vec w_{p-1}')$ where 
$\vec w_j' = \sum_i \rho^i\vec v_i\rho^{-ij} = 
\sum_i \vec v_i\rho^{-i(j+1)} = \vec w_{j+1}$ proving that 
$r$ normalizes $\hat S$. 

In more detail, 

\begin{lemma}
Let $\hat N$ be the normalizer of 
$\hat T$. Then $\hat N \cap \hat S$ 
is generated by $\sigma$. 
Similarly, if $\hat M$ is the normalizer 
of $\hat S$ then $\hat M \cap \hat T$ 
is generated by $r$.
\end{lemma}

\begin{proof}
$\hat N/\hat T$ is the symmetric group $S_p$ and 
$\sigma$ is a full $p$ cyclic in $S_p$. 
Moreover, $<\sigma>$ is the centralizer of 
$\sigma$ in $S_p$. Since $\hat S$ is abelian, 
no element outside of $<\sigma>$ in $\hat S$ can normalize $\hat T$. 
The other result is similar. 
\end{proof}

For this and other pairs of tori we will make use of the following, whose proof was provided by Gopal Prasad in a personal e-mail.  

\begin{proposition}\label{prasad}
Suppose $\hat T$ and $\hat T'$ are two maximal tori of $GL_p(\tilde F)$. 
Assume $\hat T$ and $\hat T'$ have no invariant subspaces in common. 
Then the Zariski closure of the group generated by $\hat T$ and 
$\hat T'$ is all of $GL_p(\tilde F)$. 
\end{proposition} 

\begin{proof}
We present here an adaption of Prasad's proof. 
Let $G \subset GL_p(V)$ be the closure of the 
subgroup generated by $\hat T$ and $\hat T'$. 
Clearly $G$ acts irreducibly on $V$. It follows that 
the center of $G$ is $\tilde F^* \cap G$ where 
$\tilde F^*$ is the center of $GL_p(\tilde F)$. Since 
$\tilde F^*$ is in $\hat T$ (and $\hat T'$) 
we have that $\tilde F^*$ is the center of $G$. 

Suppose $U = R_u(G)$ is the nontrivial unipotent radical of $G$.  
Then $V$ has a nontrivial submodule upon which 
$U$ acts trivially. If $W$ is the space of all vectors 
upon which $U$ acts trivially, the normality of 
$U$ implies $W$ is a $G$ module. Since $U \subset  
GL_p(V)$, $W \not= V$ and we have a contradiction. 
That is, $G$ must be a reductive group. 

From the theory of reductive groups we have that 
$G = G_1G_2$ where $G_1$ and $G_2$ commute and 
$G_2$ is almost simple. We claim $G$ must be almost 
simple. Let $V_1 \subset V$ be an 
irreducible $G_1$ module. Consider 
$V' = \sum_{g_2 \in G_2} g_2(V_1)$ As a module over $G_1$ 
this must be a finite direct sum of modules isomorphic 
to $V_1$ AND a module over $G$. 
Thus $V' = V$ and since $p$ is prime $V_1$ must be 
one dimensional or all of $V$. 
If $V$ is irreducible over $G_1$, $G_2$ must act 
centrally a contradiction. If $V_1$ is one dimensional, there is a 
$\lambda: G_1 \to F^*$ defining the action of 
$G_1$ on $V_1$ and hence $V$. Since 
$G_1$ acts faithfully, we have $G_1 = \tilde F^*$ 
which is also central and in $\hat T$ 
and $\hat T'$. Thus without loss of generality 
$G = G_2$ as needed. 

$G$ is almost simple and has an nontrivial irreducible 
representation of dimension equal to the rank. 
Using e.g. [C], this forces $G$ to be of type 
$A_n$ and hence $G = GL_p(\tilde F)$. 
\end{proof}

If $\hat T' \subset PGL_p(\tilde F)$ is a torus 
we will, without further comment, replace $\hat T'$ 
by its inverse image in $GL_p(\hat F)$ and talk about 
subgroups generated or invariant subspaces as 
is appropriate. 

To apply the above result we need to observe 
that the pairs of tori that we care about have 
the needed property. 

\begin{proposition}
1) $\hat S$ and $\hat T$ have no nontrivial common invariant subspaces. 

2) Suppose $g \in \hat S$ is not contained in the normalizer of 
$\hat T$. Then $g^{-1}\hat Tg$ and $\hat T$ have no nontrivial 
common invariant subfields. 

3) The corresponding results hold for $g \in \hat T$ 
not contained in the normalizer of $\hat S$. 
\end{proposition}

\begin{proof} 
1) If $\vec v_0,\ldots,\vec v_{p-1}$ is the basis associated to $\hat T$, then such an invariant space, $V'$, 
must be the span of a strict subset of the $\vec v_i$. 
The corresponding basis for $\hat S$ 
is $\vec w_0,\ldots,\vec w_{p-1}$ and 
$\vec w_j = \sum_i \rho^{-ij} \vec v_i$. 
All of the coefficients here are nonzero so 
none of these elements can be in $V'$. 

2) Suppose $g^{-1}\hat Tg$ and $\hat T$ 
have a nontrivial invariant subspace, $V'$, in common. 
That is, suppose $V'$ and $gV'$ are both invariant 
subspaces for $\hat T$. 
Since $\sigma$ normalizes $\hat T$ and commutes 
with $g$, $\sigma(V')$ and $g(\sigma(V'))$ 
are also invariant subspaces. 
Let $J = \{i |\vec v_i \in V'\}$. 
Applying a power of $\sigma$, we may assume 
$0 \in J$. 
Write $g = S_{\vec t}$. That is, 
$g(\vec v_j) = \sum_i t_i\vec v_{i+j}$. 
Modifying $g$ 
by a power of $\sigma$ we may assume 
$g(\vec v_0) = \sum_i t_i\vec v_i$ has 
$t_0 \not= 0$. Let $J'$ be the set of all 
$i$ with $t_i \not= 0$. Obviously 
$J' \subset J$. If $0 \not= j \in J$, 
then $g(\vec v_j) = \sum_i t_i\vec v_{i+j}$ 
and $J' + j \subset J$ for all $j \in J$.  
If $0 \not = i \in Q'$ and $j \in J$, then 
$g(\vec v_j) = \sum_i t_i\vec v_{i+j} \in V'$ 
and $i + j \in J$. That is, $i + J \subset J$.  
This implies $J = \{0,\ldots,p-1\}$ a contradiction. 
Thus $g(v_i) = t_0v_i$ for all $i$ and 
$g$ is central, another contradiction. 
\end{proof} 

Note that above we are using that 
$p$ is prime. 

A version of the $\hat{\cal T}$ and $\hat{\cal S}$ 
operations act on the other varieties we defined. 
Most simply we define ${\cal T}: P \times F[x]_p 
\to P$ by setting ${\cal T}((\alpha,\beta), f(x)) = 
(\alpha,f(\alpha)\beta)$ and ${\cal S}: 
P \times F[x]_p \to P$ by setting 
${\cal S}((\alpha,\beta), f(x)) = (f(\beta)\alpha,\beta)$. 
Note that we cannot, without making arbitrary choices, 
define ${\cal T}$ or ${\cal S}$ on $\bar P$ 
because $f(\alpha)$ and $f(\beta)$ are not homogeneous 
in $\alpha$ or $\beta$ respectively. 
However, we can define 
${\cal T}: \hat P \times \P(F[x]_p) \to \bar P$ and 
${\cal S}: \hat P \times \P(F[x]_p) \to \bar P$ 
by the same formulas. 

Observe that the difference between $\hat {\cal T}$ 
and ${\cal T}$ (and similarly for ${\cal S})$ 
is the difference between $(1/\tau)(g(x))$ and 
$f(x)$. We dignify this with a definition. 
Identify $F[x]_p = F[x]/(x^p - 1)$. 
Let $\Psi: \P(F[x]_p) \to \P(F[x]_p)$ be the rational 
map $\Psi(g(x)) = g(x)/\tau(g(x))$ defined when 
$g(x)$ is invertible. Recall that if $g(x) \in 
F[x]_p$ has associated vector $\vec g = 
(g(1),g(\rho),\ldots,g(\rho^{p-1}))$, then 
$\sigma(\vec g)$ is associated to $g(\rho{x})$. 

We define 
$\Psi'(g(x)) = g(x)/\tau'(g(x))$ similarly. 

\begin{proposition}
The image of $\Psi$ contains all $\tilde F$ points 
$h(x) \in \P(\tilde F[x]_p)$ with  $h(x)$ nonsingular. 
In particular, $\Psi$ is dominant.  
The corresponding field extension 
$k(\P(\tilde F[x]_p) \subset k(\P(\tilde F[x]_p)$ 
is $C_p$ Galois. The Galois group is realized by 
multiplication by $x$ on $\tilde F[x]/(x^p - 1)$. 
Translated to $\hat T$, $\Psi$ is just the map 
$\hat T \to \hat T/<r>$. The $\sigma$ automoprhism 
of $\hat T$ induces $f(x) \to f(\rho{x})$ 
on the image of $\Psi$. 

Similar statements hold for 
$\Psi'$. 
\end{proposition}

\begin{proof} 
If $f(x) \in \P(\tilde F[x]_p)$ we set 
$n(f) = \prod_i \tau^i(f(x))$. Then $f(x)/n(f(x))^{1/p}$ 
is in the image of $\Psi$. 
Suppose $g'(x)/\tau(g'(x)) = z(g(x)/\tau(g(x))$
for $z \in \tilde F$. 
Using the norm equals one property 
we have $z = \rho^i$ for some $i$ and 
$x^i/\tau(x^i) = \rho^i$ so 
$g'(x) = z'x^ig(x)$ for some $z' \in \tilde F^*$. 

Recall that $\hat T$ arose because 
$g(x) \in \tilde F[x]_p$ becomes the 
matrix $T_{\vec g}$ where $g_i = g(\rho^i)$. 
The Galois action associated to $\Psi$ on 
$\tilde F[x]_p$ is multiplication by $x$ 
and $(xg(x))(\rho^i) = \rho^ig(\rho^i)$. 
That is, the action of $x$ translates to the 
action of $r$. $\Psi$ becomes $\hat T \to 
\hat T/<r>$. 
\end{proof}

Arguing just as before the field extension defined 
by $\Psi$ is just 
$$k(\P(\tilde F[x]_p)(n(f(x))^{1/p})/k(\P(\tilde F[x]_p)$$
where $f(x)$ is a generic point defined on some 
affine open. 

It is obvious from the formulas that: 

\begin{lemma}\label{Qtoral}
As rational maps 
$\hat P_{\tilde F} \times \P(\tilde F[x]_p) \to 
\bar P_{\tilde F}$ we have 
$$\hat Q(\hat {\cal T}((\hat \alpha,\hat \beta),g(x)) = 
{\cal T}((\hat \alpha,\hat \beta),\Psi(g(x)))$$ 
and 
$$\hat Q(\hat {\cal S}((\hat \alpha,\hat \beta),g(x)) = 
{\cal S}((\hat \alpha,\hat \beta),\Psi'(g(x))).$$ 
\end{lemma}

We need to record the relationships between 
the $\sigma$ and $r$ operations and these toral actions. 
To this end we can employ the 
$\tau$ action on $F[x]_p$, defined above. 
Note that if $f(x) = \sum_i a_ix^i$ 
then $\tau(f(x)) = \sum_i a_i\rho^{-i}x^i$.  

Now we can easily verify: 

\begin{lemma}
a) $r(\hat{\cal T}((\hat \alpha,\hat \beta),g(x))) = 
\hat{\cal T}(r(\hat \alpha,\hat \beta),\tau(g(x)))$ 
when defined. 

b) $\sigma(\hat{\cal S}((\hat \alpha,\hat \beta),g(x)) = 
\hat{\cal S}(\sigma(\hat \alpha,\hat \beta),\tau(g(x)))$ when defined. 
\end{lemma}

\section{The Filtration} 

We are going to analyze $P$ by describing an 
increasing filitration 
$P_2 \subset \ldots \subset P_{p+1} = P$ and then analyzing each step $P_i \subset P_{i+1}$. 
As usual, we start with $\hat P_{\tilde F}$. The idea is to build 
up to $\hat P_{\tilde F}$ by applying, in alternation, the toral actions from the previous section. 
For technical reasons  
we start at $i = 2$. Since $D$ is cyclic we can fix 
the point $(\alpha_0,\beta_0)$ in $P = P(D)$ 
with image also written $(\alpha_0,\beta_0) \in \bar P$. 
We can choose $(\hat \alpha_0,\hat \beta_0) 
\in \hat P_{\tilde F}$ a preimage. 

\begin{definition}
Let $\hat P_{\tilde F,2}$ be the Zariski closure 
of 
$$\hat {\cal S}(\hat {\cal T}((\hat \alpha_0,\hat \beta_0),\P(\tilde F[x]_p)),\P(\tilde F[x]_p)).$$
\end{definition}

\begin{definition}
When $i \geq 2$ is even, define 
$\hat P_{i+1}$ to be the Zarski closure of 
$\hat {\cal T}(\hat P_i,\P(\tilde F[x]_p))$. 
When $i$ is odd, define $\hat P_{i+1}$ 
to be the Zariski closure of 
$\hat {\cal S}(\hat P_i,\P(\tilde F[x]_p))$. 
\end{definition}

Given the isomorphism $\Phi:{\cal B} \cong \hat P_{\tilde F}$ 
it is clear we have defined a filtration on ${\cal B}$ 
which we write as 
$B_2 \subset B_3 \subset \ldots \subset B_{p+1}$. 
As usual, it will help to describe the $B_i$ 
directly. Let $\Phi(\hat \alpha_0,\hat \beta_0) = 
(\vec v_0,\ldots,\vec v_{p-1}) = A$. By \ref{toral} we have 
the following lemma. 

\begin{lemma}
$B_2 = (A\hat T)\hat S$. If $i$ is even $B_{i+1}$ is the Zariski closure 
of $B_i\hat T$ and if $i$ is odd $B_{i+1}$ is the closure of $B_i\hat S$. 
\end{lemma}

One reason for starting with $i = 2$ is: 

\begin{lemma}\label{restricts}
The actions of $\sigma$ and $r$ 
on ${\cal B}$ restrict to actions on 
all the $B_i$ for $i \geq 2$. 
\end{lemma}

\begin{proof} 
$\sigma(B_2) = A\hat T\hat S\sigma = 
A\hat T\hat S$ because $\sigma \in \hat S$. 
$r(B_2) = A\hat T\hat Sr = A\hat Tr(r^{-1}\hat Sr) = 
A\hat T\hat S$ because $r \in \hat T$ and it normalizes 
$\hat S$. When $i$ is even, 
$r(B_{i+1}) = B_i\hat Tr = B_i\hat T$ and 
$\sigma(B_{i+1}) = B_i\hat T\sigma = 
\sigma(B_i)(\sigma^{-1}\hat T\sigma) = B_i\hat T$. 
A similar argument applies when $i$ is odd. 
\end{proof}

We define a parallel filtration on $P$. 

\begin{definition}
Let $P_2 \subset P$ be the Zariski closure 
of ${\cal S}({\cal T}((\alpha_0,\beta_0),F[x]_p),F[x]_p))$. For $i \geq 2$ and even define 
$P_{i+1}$ to be the Zariski closure of 
${\cal T}(P_i,F[x]_p)$ and for 
$i \geq 2$ and odd define 
$P_{i+1}$ to be the Zariski closure of 
${\cal S}(P_i,F[x]_p)$. 
\end{definition} 

In the end we will show $P_i$ is rational by induction 
on $i$. From the definition we can do the base case. 

\begin{lemma}
$k(P_2)/F$ is rational. 
\end{lemma}

\begin{proof}
From the definition we have a dominant map 
$F[x]_p \times F[x]_p \to P_2$ given 
by $f(x),h(x) \to (h(f(\alpha)\beta)\alpha,f(\alpha)\beta)$ and we want to show this map has degree 1. 
If $(h(f(\alpha)\beta)\alpha,f(\alpha)\beta) = (h'(f'(\alpha)\beta)\alpha,f'(\alpha)\beta)$ then $f(\alpha)\beta = f'(\alpha)\beta$ 
so $f(x) = f'(x)$ and $h(f(\alpha)\beta)\alpha = h'(f(\alpha)\beta)\alpha$ 
implies $h(x) = h'(x)$. 
\end{proof}

We note that all the $P_i$ are closed under the 
$F^* \times F^*$ scaling on $P$. This is another reason 
to start at $i = 2$. We set 
$\bar P_i$ to be the image of $P_i$ in $\bar P$ which by \ref{bihomogeneous} is closed.  
 
Of course we set $\bar P_{i,\tilde F} = 
\bar P_i \times_F \tilde F$ 

We have (using \ref{Qtoral}) and the fact that $\hat Q$ is finite: 

\begin{lemma} 
$\hat Q(\hat P_{i,\tilde F}) = \bar P_{i,\tilde F}$ for all $i$. 
\end{lemma}

Because of the above, we have 
Galois covers $\hat P_{i,\tilde F}/\bar P_{i,\tilde F}$ with group $C_p \oplus C_p$. We now have: 

\begin{lemma}\label{galois}
$\hat P_i/\bar P_i$ is Galois with group 
$C_p \oplus C_p$. 
\end{lemma} 

\begin{proof}
$\hat P/\bar P$ is Galois and restriction yields 
the result.
\end{proof}

The core of our argument is to understand $B_i \subset B_{i+1} = B_i\hat T$ 
(say $i$ is even for now). In this regard it is important to understand 
$\hat T_i \subset \hat T$ which is defined to be the stabilizer of 
$B_i$ in $\hat T$. That is, 
if $I_i \subset \hat F[v_{ij}]$ 
is ideal of homogeneous 
polynomials zero on $B_i$, 
$\hat T_i$ is is the stabilizer of 
$I_i$. This implies $\hat T_i \subset \hat T$ is Zariski closed. 
Since we assume $i \geq 2$, $B_i$ is left invariant by $\sigma$ 
which by \ref{restricts} implies that $\hat T_i$ is left invariant by $\sigma$. 
Similarly, if $i$ is odd and $i > 2$, 
we define $\hat S_i \subset \hat S$ as the stabilizer, then 
$\hat S_i$ is invariant under congugation by $r$. This 
severely restricts $\hat T_i$ and $\hat S_i$. 

\begin{proposition}
When $i \geq 2$ is even, $\hat T_i$ has dimension $0$ or $p-1$. If 
$\hat T_i$ has dimension $p-1$ then $B_i = {\cal B}$. 
When $i \geq 2$ is odd, $\hat S_i$ has dimension $0$ or $p-1$. 
If $\hat S_i$ has dimension $p - 1$, then $B_i = {\cal B}$. 
\end{proposition} 

\begin{proof} 
Let $C$ be the group generated by 
$\sigma$ acting on $\hat T$. 
If $M = \Hom(\hat T,\hat F^*)$ then $M$ is a $\Z[C]$ 
lattice. Since $T = (\tilde F^* \times \ldots \times \tilde F^*)/\tilde F^* = (\tilde F^*)^p/\tilde F^*$ it follows that 
$M = I[C] \subset \Z[C]$ is the kernel of 
$\Z[C] \to \Z$. Thus $M_{\Q} = M \otimes_{\Z} \Q \cong 
\Q[\rho]$ a field and so irreducible as a $\Q[C]$ module. 
If $\hat T_i$ has positive dimension, there is an induced 
$(M \otimes_{\Z} \Q) \to (\Hom(\hat T_i,\hat F^*) \otimes_{\Z} \Q)$ which is a surjective $C$ morphism implying $\Hom(\hat T_i, \hat F^*) \otimes_{\Z} \Q$ 
has dimension $p-1$ and so $\hat T_i$ has dimension $p-1$. 
If $\hat T_i$ has dimension $p - 1$ then $\hat T_i = \hat T$ 
and so $\hat T$ acts on $B_i$. 
Since $B_i$ was defined as the closure under the 
$\hat S$ action, $B_i$ is closed under both $\hat T$ and $\hat S$ 
actions. By (\ref{prasad}) these generate $GL_p$ up to closure, so 
$B_i = {\cal B}$. 

A similar argument applies to $\hat S_i$.  
\end{proof} 

Thus when $i \leq p$ we know that $\hat T_i$ (or 
$\hat S_i$) are zero dimensional and thus finite. 

We can be more precise about the meaning of these stabilizers. 

\begin{lemma}
Let $i \leq p$ be even. There is an open subset $U \subset B_i$ 
such that for all $x \in B_i$, $\hat T_i = \{t \in \hat T | xt \in B_i\}$. 
A similar results holds for $i$ odd. 
\end{lemma}

\begin{proof} 
Define $Z \subset B_i \times \hat T$ as 
$\{(x,t) | xt \in B_i\}$. $Z$ is clearly closed and we give it the reduced structure`. 
Also, since $\hat T$ contains $1$, 
the restricted projection $\psi: Z \to B_i$ 
is surjective. The generic fiber is, by definition, 
$\hat T_i$. Thus $\psi$ is generically finite and 
hence generically etale. 
\end{proof}

\begin{corollary}
If $i \leq p$ then the dimension of $B_{i+1}$ 
is $p - 1$ plus the dimension of $B_i$. 
$B_{p+1} = {\cal B}$. 
\end{corollary}

\begin{proof}
If $U \subset B_i$ is the open set from the lemma, 
then $\phi: U \times \hat T \to B_{i+1}$ contains an open $V$ 
in the image. If $xt \in V$ then $\phi^{-1}(xt) = \{(x',t') | xt = x't'\} 
\cong \{t | xt \in B_i\} = \hat T_i$. Thus $\phi$ 
has generic fiber dimension $0$ implying $B_{i+1}$ 
has dimension $p - 1$ plus the dimension of $B_i$. 
Adding we have 
that the dimension of $B_{p+1}$ is equal to the dimension of 
${\cal B}$ proving the result. 
\end{proof} 

Using \ref{galois} we have: 

\begin{corollary}
$\bar P_i$ has dimension $i(p-1)$ and $\bar P_{p+1} = \bar P$. 
$P_i$ has dimension $i(p - 1) + 2$ and $P_{p+1} =P$.  
\end{corollary}

We can be more specific about the $\hat T_i$ 
and $\hat S_i$. 

\begin{proposition}
Let $2 \leq i \leq p$. 
If $i$ is even, the stabilizer $\hat T_i$ is generated by $r$ and hence has order $p$. 
On an open subset of $x \in B_i$, $<r> = \{t \in \hat T | tx \in B_i\}$. 
If $i$ is odd, 
$\hat S_i$ is generated by $\sigma$ and hence has order 
$p$. On an open subset of $x \in B_i$, 
$<\sigma> = \{s \in \hat S | sx \in B_i\}$. 
\end{proposition}

\begin{proof}
We detail the $i$ even case as the other case is parallel. Suppose $t \in \hat T_i$ is not 
in the group generated by $r$. Since $B_i$ 
is closed under the operation of $\hat S$, it closed 
under the operations of $\hat S$ and $t\hat St^{-1}$. 
Since $t$ does not normalize $\hat S$, $B_i$ 
is closed under the operation of all of 
$GL_p(\tilde F)$ (\ref{prasad}), a contradiction. 
\end{proof} 

We can now be more precise about the morphism 
$B_i \times \hat T \to B_{i+1}$ ($i$ even) 
or $B_i \times \hat S \to B_{i+1}$ ($i$ odd). 

\begin{proposition}
Let $p + 1 \geq i \geq 2$ be even. Let 
$C = <r>$ act on $B_i \times \hat T$ via 
$r(x,t) = (xr^{-1},rt)$. Then $k(B_{i+1}) = 
k(B_i \times \hat T)^C$. If $i$ is odd and 
$C = <\sigma>$ acts on $B_i \times \hat S$ via 
$\sigma(x,s) = x\sigma^{-1},\sigma{x})$ 
then $k(B_{i+1}) = k(B_i \times \hat S)^C$. 
\end{proposition}

\begin{proof} 
As usual we need only prove the $i$ even case. 
Since $xrr^{-1}t = xt$ it is clear that 
$B_i \times \hat T$ induces $(B_i \times \hat T)/C 
\to B_{i+1}$. Also, if $xt = x't'$ then 
$x' = xt(t')^{-1}$ and so on an open subset of 
$B_i$ we have $t'(t)^{-1} = r^m \in T_i$ \and 
$x' = xr^{-m}$ and $t' = r^mt$. That is, 
this map has degree one and is a birational 
isomorphism. 
\end{proof}

\section{Rationality} 

We restate the rationality result that seems to be known as 
the "no-name" lemma. 
Recall that if $L/K$ is $G$ Galois 
and $V$ is a finite dimensional $L$ vector space we say $G$ acts semilinearly on $V$ if $G$ acts and 
for all $\ell \in L$,$v \in V$, and 
$g \in G$  
we have $g(\ell{v}) = g(\ell)g(v)$. 
Recall that 
$k(V)$ is the field of fractions of the 
$L$ vector space $V$ and $k(V^G)$ is the field of fractions of the $K$ vector space 
$V^G$. 
The proof follows from the Galois 
descent fact that $LV^G = V$.

\begin{lemma}(\cite{EM} p. 16)
Let $G$ be a finite group and $L/K$ a Galois extension of fields. 
Assume $V$ is an $L$ vector space with a semilinear action 
by $G$. Then $k(V)^G = k(V^G)$ which is therefore rational 
over $K$. 
\end{lemma}

We have already observed that $P_2$ is rational. 
Thus to prove $P$ rational it suffices to show 
that $k(P_{i+1})/k(P_i)$ is rational for all $2 \leq i \leq p$. 
To show this it suffices to show that 
$k(\bar P_{i+1})/k(\bar P_i)$ is rational. 

In fact we will show the following. We assume $i$ is even 
as the other case is parallel. 
On $\hat P_i \times \P(F[x]_p)$ we let $r$ act via $r' = r \times 1$
but we let $\sigma$ act as $\sigma' = \sigma \times \sigma$. 

Let us begin by proving the last step of the argument for the above, which will serve 
to explain the method. 

\begin{theorem}
If $\bar P_{i+1}$ is birationally isomoprhic to 
$(\hat P_i \times \P(F[x]_p))/<r',\sigma'>$ then 
$k(\bar P_{i+1})/k(\bar P_i)$ is rational. 
\end{theorem}

\begin{proof}
Let $K_j = k(\hat P_j)$ and $L_j = k( K_j \times_F \P(F[x]_p))$. Translated to fields 
we have $r'$ acting on $L_i$ via its 
action on $K_i$ and $\sigma'$ acting by acting on $K_i$ and $\P(F[x]_p)$ 
diagonally. The assumption of the theorem is that the fixed field 
$L_i^{<r',\sigma'>} = K_{i+1}$. 
But the $r'$ invariant field is clearly 
$k(K_i^{<r>} \times \P(F[x]_p))$. Now we saw that $\sigma$ acts 
on $F[x]_p$ via $f(x) \to f(\rho{x})$ which is a linear action. 
In fact letting $1,x,\ldots,x^{p-1}$ be the basis of $F[x]_p$, then 
$x^i/1 \in K_i^{<r>}(\P(F[x]_p)$ is an eigenvectorwith eigenvalue 
$\rho^i$. 
By the "no-name" lemma $k((K_i^{<r>}\times \P(F[x]_p))^{<\sigma'>}$ 
is rational over $K_i^{<r,\sigma>} = k(\bar P_i)$. 
\end{proof}

Obviously we want to prove: 

\begin{theorem} 
$k(K_i \times \P(F[x]_p)^{<r',\sigma'>} = k(\bar P_{i+1})$. 
\end{theorem} 

To prove the above equality it suffices to prove it after extending 
scalars to $\tilde F$. That is, it suffices to show the following. 
Let $\tilde K_j = k(\hat P_{\tilde F})$ and 
$L_j = k(\tilde K_j \times_{\tilde F} \P(\tilde F[x]_p))$. 
We claim: 

\begin{theorem}
$k(\tilde K_i \times \P(\tilde F[x]_p)^{<r',\sigma'>} = k(\bar P_{i+1,\tilde F})$
\end{theorem} 

The above is a result about $\hat P_{i,\tilde F} \times \P(\tilde F[x]_p) 
\to \bar P_{i+1,\tilde F}$. We translate the statement to 
the $B_i \subset {\cal B}$. For convenience we only detail the case 
$i$ is even, as the other case is parallel. We have 
$(B_i \times \hat T/(<r,r^{-1}>)  = B_{i+1}$ and $B_{i+1}/(<r,\sigma>) = 
\bar P_{i+1,\tilde F}$. Now $r$ acts on $B_{i+1}$ by right multiplication. 
That is, we can view it as acting on $\hat T$. The group generated by 
$(r,r^{-1})$ and $r$ is $<r> \times <r>$ and so quotienting 
by the $r$ stuff we get $(B_i/<r>) \times (\hat T/<r>)$. Now 
$\hat T/<r>$ is birationally $\P(F[x]_p)_f$ via $\Psi$. 
$\sigma$ acts by right multiplication and if $A \in B_i$ we have 
$AT\sigma = (A\sigma)(\sigma^{-1}T\sigma)$ so $\sigma$ acts 
diagonally on $B_i\hat T$. Altogether, $\bar P_{i+1,\tilde F}$ 
is $B_i/r \times \P[F[x]_p)_f/(\sigma,\sigma)$ as claimed 
and needed. This proves $P_{p+1} = P$ is rational. 

\section{Final Result} 

\begin{theorem}\label{result}
If $UD(F,p)$ is cyclic then $Z(F,p)$ is stably rational. 
\end{theorem}

\begin{corollary}
If $Z(F,p)$ is not stably rational then $UD(F,p)$ is not cyclic. 
\end{corollary}

\begin{proof}
Let $P = P(UD(F,p))$ and $K = k(P)$ which we showed is rational 
over $Z(F,p)$. We saw in \ref{iso} that $K$ has the following alternate 
description. Let $F' = F(x,y)$ purely transcendental of degree 2 
over $F$ and let $E = (x,y)_{p,F'}$. Then $K$ is the field of fractions 
of $SB(UD \otimes_{Z(F,p)} E^{\circ}_{Z(F,p)})$. By 
\ref{rational}, $K/F'$ and hence $K/F$ is rational.
\end{proof} 

The argument of this paper actually allows us to prove something 
slightly stronger, and perhaps useful, which we state as a corollary. 

\begin{corollary}
Suppose $Z(F,p)(x_1,\ldots,x_{p^2+1})$ is the purely transcendental 
extension of $Z(F,p)$ of degree $p^2 + 1$. If 
$Z(F,P)(x_1,\ldots,x_{p^2+1})$ is not rational, 
$UD(F,p)$ is not cyclic. 
\end{corollary}

We end this section with a note about something we have not proven. 
If $P(D) = P$ contains $(\alpha,\beta)$ 
with $\alpha^p = x$ and $\beta^p = y$, then 
$(f(\beta)\alpha)^p = N_{F(\beta)/F}(f(\beta))x$ 
and $(f(\alpha)\beta)^p = N_{F(\alpha)/F}(f(\alpha))y$. 
This reminds one of the so called common slot 
relations 
$$(x,y)_{p,F} \cong (N_{F(y^{1/p})/F}(z)x,y)_{p,F}  
\cong (x,N_{F(x^{1/p})/F}(z)y)_{p,F}.$$ 
In fact, the machinery above defines a dominant 
morphism $\Delta: F[x]_p \times \ldots \times F[x]_p \to P$ 
whose image are the pairs reached by this 
common slot process. If $D \cong (x',y')_{p,F}$ 
then the pairs $(\alpha',\beta') \in P$ with 
$\alpha'^p = x'$ and $\beta'^p = y'$ are dense 
in $P$ but that does NOT imply one of them is in the 
image of $\Delta$.

\section{Appendix}

The arguments above do not require this section, 
but it may be useful to note that the definition 
of $P$ does not require taking radicals and that 
we can directly prove that $P$ is smooth (which could already 
be deduced from the pervious sections but not directly). 

To set up what I mean by all this, let $D/F$ be cyclic of degree 
$p$ as above 
but we do NOT need $F$ of characteristic 0 just that the 
characteristicof $F$ is prime to $p$ (and containing $\rho$, a primitive $p$ root of one). Let 
$u_1,\ldots,u_{p^2}$ be an $F$ basis of $D$ and let 
$\vec x = (x_1,\ldots,x_{p^2})$, $\vec y = (y_1,\ldots,y_{p^2})$ 
be vectors of indeterminants. In $D \otimes_F F[\vec x,\vec y]$ 
we have "generic" elements $X = \sum_i x_iu_i$ and 
$Y = \sum_j y_iu_i$. Write $XY - \rho{YX} = \sum_i f_i(\vec x,\vec y)u_i$. 
Let $n(X)$, $n(Y)$ be the norms of $X$ and $Y$, both elements of 
$F[\vec x,\vec y]$. Let 
$S =  F[\vec x,\vec y](1/(n(X)n(Y)))$. 
and let $I \subset S$ be the ideal generated by the 
$f_i$. Set $R = S/I$. The point of this appendix is 
to show: 

\begin{theorem}
$R$ is formally smooth and hence smooth. 
\end{theorem} 

In particular the above shows $I$ reduced and we need not 
take radicals. 

Note that if we change the basis $u_i$  over $F$ then this just makes 
linear changes in the variables of the $f_i$ and so does not change the 
result. Thus we can and will assume that $u_1$ has trace $1$ and all the  
$u_i$ with $i > 1$ have trace 0. 

To prove this theorem we start with some observations 
about skew commuting pairs in Azumaya algebras over arbitrary 
(non-reduced) $F$ algebras. 

\begin{theorem}
Suppose $R'$ is a commutative $F$ algebra and $A/R'$ is an Azumaya 
algebra of degree $p$. Assume $A$ contains invertible $\alpha,\beta$ 
such that $\alpha\beta = \rho\beta\alpha$. 
Then $A = \sum_{i,j = 0}^{p-1} R'\alpha^i\beta^j$ 
and $\alpha^p, \beta^p \in (R')^*$. That is, $A$ is an 
Azumaya symbol algebra. 
\end{theorem}

To start with we observe: 

\begin{lemma}
In $M_{p-1}(F)$ the matrix $R_1$ with $i,j$ entry 
$\rho^{ij} - 1$ is invertible. 
\end{lemma}

In the above lemma, note that we are labelling the rows and 
columns of $R_1$ by $1,\ldots,p-1$. 
The proof will make clear why. 

\begin{proof} 
We start with the $p \times p$ matrix $R = (\rho^{ij})$ 
which we know is invertible. From each row $1,\ldots,p-1$ 
we subtract the first row to get a matrix $R''$. 
The first (i.e $0$) column is 
$1,0,\ldots,0$ and when we expand the determinant 
of $R''$ using that first column we get that $R_1$ has nonzero determinant. 
\end{proof}

\begin{proof}
Returning to the proof of the second theorem, let 
$\alpha$, $\beta$ be as given. Since $A$ is Azumaya 
of degree $p$, we have the characteristic polynomial 
equation $\alpha^p + s_1\alpha^{p-1} + \ldots s_{p-1}\alpha + s_p = 0$. 
Conjugating this equation by $\beta^j$ we have 
 $\alpha^p + s_1\rho^{j(p-1)}\alpha^{p-1} + \ldots + 
s_{p-1}\rho^j\alpha + s_p = 0$.  We can subtract the $j = 0$ 
equation from the $j = 1,\ldots,p-1$ equations eliminating 
the $\alpha^p$ and $s_p$ terms. Applying the above 
lemma we get $s_j\alpha^{p-j} = 0$ which implies $s_j = 0$ 
for $j = 1,\ldots,p-1$. That is the canonical equation 
for $\alpha$ is $\alpha^p + s_p = 0$ and we know $s_p$ 
is invertible. We have the parallel result for $\beta$. 

Let $A' = \sum_{i,j = 0}^{p-1} R'\alpha^i\beta^j \subseteq A$.  
Since $A'$ is an Azumaya symbol algebra, it is Azumaya 
of degree $p$. By a standard argument (e.g. \cite{S} p. 15) 
$A' = A$. 
\end{proof}

\begin{proof}
We now can turn to the proof of the main, first theorem. 
If $T$ is a commutative $F$ algebra and $J \subset T$ 
is a nilpotent ideal, we must show that any $F$ algebra 
homomorphism $\phi': R \to T/J$ lifts to $\phi: R \to T$. 
By the usual induction argument, we can assume 
$J^2 = 0$. it is now clear that we must prove: 

\begin{proposition}
Suppose $\alpha,\beta \in (D \otimes_F (T/J))^*$ 
satisfy $\alpha\beta = \rho\beta\alpha$, then there are 
$\alpha' \beta' \in D \otimes_F T$ which are preimages 
of $\alpha$, $\beta$ and which satisfy 
$\alpha'\beta' = \rho\beta'\alpha'$. 
\end{proposition}
 
\begin{proof}
Let $\alpha', \beta' \in D \otimes_F T$ be arbitrary 
lifts of $\alpha$, $\beta$. As such, they are automatically 
invertible and $\alpha'\beta' - \rho\beta'\alpha' = z \in J(D \otimes_F T)$. 
Note that the left side of the above equation has trace 0 
and so $z$ has trace zero. 

Let $x ,y \in J(D \otimes_F T)$ 
be arbitrary and consider 
$$\alpha'(1 + x)\beta'(1 + y) - 
\rho\beta'(1 + y)\alpha'(1 + x) = $$
$$\alpha'\beta'(1 + \beta'^{-1}x\beta')(1 + y) - 
\rho\beta'\alpha'(1 + \alpha'^{-1}y\alpha)(1 + x) = $$
$$\alpha'\beta'(1 + \beta'^{-1}x\beta')(1 + y) - [(\alpha'\beta') - z]
(1 + \alpha'^{-1}y\alpha')(1 + x) = $$ 
$$\alpha'\beta'(1 + \beta'^{-1}x\beta') + y) - 
[\alpha'\beta' - z]( 1 + \alpha'^{-1}y\alpha + x)$$ 
since $J^2 = 0$. 
Simplifying further the above equals:
$$\alpha'\beta'((\beta'^{-1}x\beta - x)+  (y - \alpha'^{-1}y\alpha' ) - z).$$
Thus the proposition is proven if we show: 

\begin{lemma}
Every trace $0$ element of $J(D \otimes_F T)$ has the form 
$(\beta'^{-1}x\beta' - x)+  (y - \alpha'^{-1}y\alpha' )$ 
for some choice of $x,y \in J(D \otimes_F T)$. 
\end{lemma}

\begin{proof}
First of all, since $J^2 = 0$, we note that 
$\alpha'^{-1}y\alpha'$ and $\beta'^{-1}x\beta'$ are independent 
of the choice of $\alpha'$ and $\beta'$. 
The trace $0$ elements of $J(D \otimes_F T)$ 
are $Ju_2 + \ldots + Ju_{p^2}$. 
Clearly is suffices to show 
that for any $r \in J$ and any $u_i$ with $i > 1$ then $ru_i$ can be 
written in that form. That is, it suffices to show that 
any trace 0 element of $D \otimes_F T/J$ can be written as 
$\beta^{-1}x'\beta - x' + y' - \alpha^{-1}y'\alpha$ 
for some $x',y' \in D \otimes_F T/J$. 

We are reduced to showing that if $T'$ is a commutative 
$F$ algebra, $A'/T'$ is a degree $p$ Azumaya symbol algebra 
generated by skew commuting nonsingular $\alpha,\beta$, 
and $\phi: A' \oplus A' \to A'$ is defined by 
$\phi'(x',y') = \beta^{-1}x\beta - x + y - \alpha^{-1}y\alpha$ 
then $\phi$ is surjective onto the trace 0 elements. 

Now in $A'$ the trace 0 elements are spanned by $\alpha^i\beta^j$ 
for $(i,j) \not= (0,0)$. 
$\beta^{-1}\alpha^i\beta^j\beta - \alpha^i\beta^j = 
(\rho^i - 1)\alpha^i\beta^j$ so in the image of $\phi$ 
are all $\alpha^i\beta^j$ for $i > 0$. Similarly, 
in the image of $\phi$ are all $\alpha^i\beta^j$ for $j > 0$. 
This proves the claim, the proposition, and the main theorem. 

\end{proof}

\end{proof}

\end{proof}

Let us note that we can extend the above 
argument to directly show $\hat P$ is smooth. That is, if we add to $I$ above 
the equations expressing $X^p = 1$, 
$Y^p = 1$, and do not take the radical, 
then $\hat P$ defined in this way is 
smooth and hence already reduced. 

We follow the same outline, and show 
we can lift skew commuting pairs 
$\alpha$, $\beta$ where $\alpha^p = 1 = \beta^p$. We lift as above, and now 
$\alpha'^p = 1 + x$ and $\beta'^p = 
1 + y$ where $x,y \in J$. But as the characteristic of $F$ is prime to $p$, 
$1 + x$ and $1 + y$ have $p$ roots.

\end{document}